\documentclass[12pt]{amsart}
\usepackage{amscd,amssymb,latexsym}
\usepackage{fullpage}
\input{xypic}
\newcommand{\Mdef}[2]{\newcommand{#1}{\relax \ifmmode #2 \else $#2$\fi}}

\usepackage{geometry}
\usepackage{url}
\usepackage[svgnames]{xcolor} 
\usepackage{hyperref}
  \definecolor{dark-red}{rgb}{0.4,0.15,0.15}

\newcommand{\bft}{\mathbf{t}}
\newcommand{\bfT}{\mathbf{T}}

\newcommand{\sm }{\wedge}

\newcommand{\tensor}{\otimes}

\newcommand{\Hom}{\mathrm{Hom}}

\newcommand{\Ext}{\mathrm{Ext}}
\Mdef{\bhom}{\mathbf{\hat{H}om}}

\Mdef{\Mod}{\mathrm{mod}}

\newcommand{\st}{\; | \;}

\newtheorem{thm}{Theorem}[section]
\newtheorem{lemma}[thm]{Lemma}
\newtheorem{prop}[thm]{Proposition}

\theoremstyle{definition}

\newtheorem{defn}[thm]{Definition}

\newtheorem{example}[thm]{Example}

\newtheorem{remark}[thm]{Remark}

\newcommand{\qqed}{\qed \\[1ex]}
\renewenvironment{proof}[1][\hspace*{-.8ex}]{\noindent {\bf Proof #1:\;}}{\qqed}

\Mdef{\PH} {\Phi^H}
\Mdef{\PK} {\Phi^K}
\Mdef{\PL} {\Phi^L}
\Mdef{\PT} {\Phi^{\T}}

\Mdef{\ef}{E{\cF}_+}
\Mdef{\etf}{\widetilde{E}{\cF}}
\Mdef{\eg}{E{G}_+}
\Mdef{\etg}{\tilde{E}{G}}

\Mdef{\infl}{\mathrm{inf}}
\Mdef{\defl}{\mathrm{def}}
\Mdef{\res}{\mathrm{res}}
\Mdef{\ind}{\mathrm{ind}}
\Mdef{\coind}{\mathrm{coind}}

\Mdef{\univ}{\mathcal{U}}

\Mdef{\Fp}{\mathbb{F}_p}
\Mdef{\Zpinfty}{\Z /p^{\infty}}
\Mdef{\Zpadic}{\Z_p^{\wedge}}

\newcommand{\bi}{\begin{itemize}}
\newcommand{\be}{\begin{enumerate}}
\newcommand{\bc}{\begin{center}}
\newcommand{\bd}{\begin{description}}
\newcommand{\ei}{\end{itemize}}
\newcommand{\ee}{\end{enumerate}}
\newcommand{\ec}{\end{center}}
\newcommand{\ed}{\end{description}}

\newcommand{\lra}{\longrightarrow}
\newcommand{\lla}{\longleftarrow}

\Mdef{\we}{\mathbf{we}}
\Mdef{\fib}{\mathbf{fib}}
\Mdef{\cof}{\mathbf{cof}}
\Mdef{\BI}{\mathcal{BI}}

\newcommand{\fibre}{\mathrm{fibre}}

\newcommand{\hocolim}{\mathop{  \mathop{\mathrm {holim}}\limits_\rightarrow} \nolimits}

\Mdef{\A}{\mathbb{A}}
\Mdef{\B}{\mathbb{B}}
\Mdef{\C}{\mathbb{C}}
\Mdef{\D}{\mathbb{D}}
\Mdef{\E}{\mathbb{E}}
\Mdef{\T}{\mathbb{T}}
\Mdef{\F}{\mathbb{F}}
\Mdef{\G}{\mathbb{G}}
\Mdef{\I}{\mathbb{I}}
\Mdef{\N}{\mathbb{N}}
\Mdef{\Q}{\mathbb{Q}}
\Mdef{\R}{\mathbb{R}}
\Mdef{\bbS}{\mathbb{S}}
\Mdef{\Z}{\mathbb{Z}}

\Mdef{\bA}{\mathbb{A}}
\Mdef{\bB}{\mathbb{B}}
\Mdef{\bC}{\mathbb{C}}
\Mdef{\bD}{\mathbb{D}}
\Mdef{\bE}{\mathbb{E}}
\Mdef{\bF}{\mathbb{F}}
\Mdef{\bG}{\mathbb{G}}
\Mdef{\bH}{\mathbb{H}}
\Mdef{\bI}{\mathbb{I}}
\Mdef{\bJ}{\mathbb{J}}
\Mdef{\bK}{\mathbb{K}}
\Mdef{\bL}{\mathbb{L}}
\Mdef{\bM}{\mathbb{M}}
\Mdef{\bN}{\mathbb{N}}
\Mdef{\bO}{\mathbb{O}}
\Mdef{\bP}{\mathbb{P}}
\Mdef{\bQ}{\mathbb{Q}}
\Mdef{\bR}{\mathbb{R}}
\Mdef{\bS}{\mathbb{S}}
\Mdef{\bT}{\mathbb{T}}
\Mdef{\bU}{\mathbb{U}}
\Mdef{\bV}{\mathbb{V}}
\Mdef{\bW}{\mathbb{W}}
\Mdef{\bX}{\mathbb{X}}
\Mdef{\bY}{\mathbb{Y}}
\Mdef{\bZ}{\mathbb{Z}}

\Mdef{\cA}{\mathcal{A}}
\Mdef{\cB}{\mathcal{B}}
\Mdef{\cC}{\mathcal{C}}
\Mdef{\mcD}{\mathcal{D}} 
\Mdef{\cE}{\mathcal{E}}
\Mdef{\cF}{\mathcal{F}}
\Mdef{\cG}{\mathcal{G}}
\Mdef{\mcH}{\mathcal{H}} 
\Mdef{\cI}{\mathcal{I}}
\Mdef{\cJ}{\mathcal{J}}
\Mdef{\cK}{\mathcal{K}}
\Mdef{\mcL}{\mathcal{L}}

\Mdef{\cM}{\mathcal{M}}
\Mdef{\cN}{\mathcal{N}}
\Mdef{\cO}{\mathcal{O}}
\Mdef{\cP}{\mathcal{P}}
\Mdef{\cQ}{\mathcal{Q}}
\Mdef{\mcR}{\mathcal{R}}
\Mdef{\cS}{\mathcal{S}}
\Mdef{\cT}{\mathcal{T}}
\Mdef{\cU}{\mathcal{U}}
\Mdef{\cV}{\mathcal{V}}
\Mdef{\cW}{\mathcal{W}}
\Mdef{\cX}{\mathcal{X}}
\Mdef{\cY}{\mathcal{Y}}
\Mdef{\cZ}{\mathcal{Z}}

\Mdef{\ca}{\mathcal{a}}
\Mdef{\ct}{\mathcal{t}}

\Mdef{\At}{\tilde{A}}
\Mdef{\Bt}{\tilde{B}}
\Mdef{\Ct}{\tilde{C}}
\Mdef{\Et}{\tilde{E}}
\Mdef{\Ht}{\tilde{H}}
\Mdef{\Kt}{\tilde{K}}
\Mdef{\Lt}{\tilde{L}}
\Mdef{\Mt}{\tilde{M}}
\Mdef{\Nt}{\tilde{N}}
\Mdef{\Pt}{\tilde{P}}

\Mdef{\tA}{\tilde{A}}
\Mdef{\tB}{\tilde{B}}
\Mdef{\tC}{\tilde{C}}
\Mdef{\tE}{\tilde{E}}
\Mdef{\tH}{\tilde{H}}
\Mdef{\tK}{\tilde{K}}
\Mdef{\tL}{\tilde{L}}
\Mdef{\tM}{\tilde{M}}
\Mdef{\tN}{\tilde{N}}
\Mdef{\tP}{\tilde{P}}

\Mdef{\ft}{\tilde{f}}
\Mdef{\xt}{\tilde{x}}
\Mdef{\yt}{\tilde{y}}

\Mdef{\Ab}{\overline{A}}
\Mdef{\Bb}{\overline{B}}
\Mdef{\Cb}{\overline{C}}
\Mdef{\Db}{\overline{D}}
\Mdef{\Eb}{\overline{E}}
\Mdef{\Fb}{\overline{F}}
\Mdef{\Gb}{\overline{G}}
\Mdef{\Hb}{\overline{H}}
\Mdef{\Ib}{\overline{I}}
\Mdef{\Jb}{\overline{J}}
\Mdef{\Kb}{\overline{K}}
\Mdef{\Lb}{\overline{L}}
\Mdef{\Mb}{\overline{M}}
\Mdef{\Nb}{\overline{N}}
\Mdef{\Ob}{\overline{O}}
\Mdef{\Pb}{\overline{P}}
\Mdef{\Qb}{\overline{Q}}
\Mdef{\Rb}{\overline{R}}
\Mdef{\Sb}{\overline{S}}
\Mdef{\Tb}{\overline{T}}
\Mdef{\Ub}{\overline{U}}
\Mdef{\Vb}{\overline{V}}
\Mdef{\Wb}{\overline{W}}
\Mdef{\Xb}{\overline{X}}
\Mdef{\Yb}{\overline{Y}}
\Mdef{\Zb}{\overline{Z}}

\Mdef{\db}{\overline{d}}
\Mdef{\hb}{\overline{h}}
\Mdef{\qb}{\overline{q}}
\Mdef{\rb}{\overline{r}}
\Mdef{\tb}{\overline{t}}
\Mdef{\ub}{\overline{u}}
\Mdef{\vb}{\overline{v}}

\Mdef{\hc}{\hat{c}}
\Mdef{\he}{\hat{e}}
\Mdef{\hf}{\hat{f}}
\Mdef{\hA}{\hat{A}}
\Mdef{\hH}{\hat{H}}
\Mdef{\hJ}{\hat{J}}
\Mdef{\hM}{\hat{M}}
\Mdef{\hP}{\hat{P}}
\Mdef{\hQ}{\hat{Q}}

\Mdef{\thetab}{\overline{\theta}}
\Mdef{\phib}{\overline{\phi}}

\Mdef{\uA}{\underline{A}}
\Mdef{\uB}{\underline{B}}
\Mdef{\uC}{\underline{C}}
\Mdef{\uD}{\underline{D}}

\Mdef{\bolda}{\mathbf{a}}
\Mdef{\boldb}{\mathbf{b}}
\Mdef{\bfD}{\mathbf{D}}

\Mdef{\fm}{\frak{m}}
\Mdef{\fp}{\frak{p}}

\Mdef{\eps}{\epsilon}

\newcommand{\cell}{\mathrm{Cell}}

\newcommand{\bbI}{\mathbb{I}}
\newcommand{\Cech}{\v{C}ech}

\newcommand{\fn}{\mathfrak{n}}
\newcommand{\Gn}{\Gamma_{\fn}}
\newcommand{\Cn}{\check{C}_{\fn}}

\newcommand{\bbF}{{\mathbb{F}}}

\newcommand{\coext} [3] {#1\!\!\Uparrow_{#2}^{#3}}
\newcommand{\Rmod}{\mbox{$R$-mod}}

\newcommand{\Abstar}{\mathrm{Ab}_*}

\newcommand{\HKt}{\widetilde{HK}}
\newcommand{\rc}{{\mathbf{r}}}
\renewcommand{\Rc}{{\mathbf{R}}}
\newcommand{\bfx}{\mathbf{x}}

\begin{document}
\title{Anderson and Gorenstein duality}
\author{J.P.C.Greenlees}
\address{School of Mathematics and Statistics, Hicks Building, 
Sheffield S3 7RH. UK.}
\email{j.greenlees@sheffield.ac.uk}
\author{V.Stojanoska}
\address{Department of Mathematics, University of Illinois, 1409 W Green St., Urbana IL 61802, USA.}
\email{vesna@illinois.edu}
\date{}

\begin{abstract}
The paper relates the Gorenstein duality
statements of \cite{DGI, DGI3} to the Anderson duality statements of 
\cite{Stojanoska1, Stojanoska2}, and explains how to use  
local cohomology and invariant theory to understand the numerology of 
shifts in simple cases.
\end{abstract}

\thanks{We are grateful to MSRI and MPI for giving us the opportunity to
  start these discussions,  to the referee for careful reading and
  detailed comments, to J.Rognes for an email conversation suggesting the connection
  described in Subsection \ref{subsec:MR}, and to C.Rezk about further discussion regarding that connection. The second author thanks the NSF for support through grant DMS-1606479.}
\maketitle

\tableofcontents
\section{Introduction}

\subsection{Motivation}
This paper emerged from a desire to understand the relationship
between the duality statements that the two authors had been working
on.  More precisely, we wished to relate the Gorenstein duality
statements of \cite{DGI, DGI3} to the Anderson duality statements of 
\cite{Stojanoska1, Stojanoska2}.  It was clear they were closely
related, but here we make the relationship precise.

One of us had been considering connective ring spectra $\rc$ (such as
$ku$ or $tmf_1(n)$) and proving when they have Gorenstein duality, and
one of us had been considering non-connective spectra $\Rc$ 
(such as $KU$ or $Tmf_1(n)$) and proving when they are Anderson
self-dual. In many cases of interest, it is easy to recover $\rc$ as
the connective cover of $\Rc$, 
but also in favourable cases $\Rc$ can be recovered from $\rc$ by a well known 
localization process, and under these processes the dualities
correspond
(Proposition \ref{prop:GorDAndD} and Lemma \ref{lem:AndDGorD}).

In many cases $\rc$ can  immediately be seen to have Gorenstein 
duality since  the coefficient ring $\rc_*$ has it. Similar reasoning on the level of coefficients then gives that
$\Rc$ is Anderson self-dual. 

 Furthermore in many cases there is a  Galois-like action of a finite
 group  $G$ on $\rc$  and on $\Rc$, which is compatible with the process
of moving from $\rc$ to $\Rc$ and back again.  (In the
above cases $G$ is $C_2$ or $(\Z/n)^{\times}$). In favourable  cases the fixed point ring spectra $\rc^G$ are of interest
($ko$ or $tmf_0(n)$). Furthermore, the action of 
$G$ on $\Rc$ is Galois with fixed point spectrum equal to the
homotopy fixed point spectrum, and $\Rc^G\simeq \Rc^{hG}$ is  also 
of interest ($KO$ or $Tmf_0(n)$). It may happen that the Gorenstein
duality of $\rc$ descends to that of $\rc^G$, or that the Anderson self-dualiy  of $\Rc$ descends to that of $\Rc^G$, but even when this happens
the shift will change. 

The simplest case is when the group order is invertible, so that the
coefficients of the homotopy fixed points are the invariants: 
$\Rc^{hG}_*=( \Rc_*)^G$, and we point out here that in this case character
theory often predicts the change in shift. 

In general these examples come in fours: $\rc, \Rc, \rc^G$ and
$\Rc^G$. One may hope to prove duality (in cases where it holds) by
the following route:
(1) we have duality  for $\rc_*$ and hence for $\rc$ (2) we infer duality
for $\Rc$ (3) we obtain duality for $\Rc^G=\Rc^{hG}$ by descent and (4) we
infer duality for $\rc^G$. The contents of this paper deal with the
step from  (1) to (2) and from  (3) to (4). The step from (2) to (3)
is more subtle and more interesting, and we hope to return to it
elsewhere. The interested reader can find specific examples of this step in \cite{HS,HillMeier,Stojanoska2}; related is the step (1) to (4), worked out in specific examples in \cite{GM}.

Beyond $K$-theory, our examples come from derived algebraic
geometry. In this setting, it is the spectra $\Rc$, rather than the
connective $\rc$, which are primordial. In the presence of a gap in
the homotopy groups of $\Rc$, one gets $\rc$ as the connective
cover. Unfortunately, there is no known procedure for obtaining $\rc$
from $\Rc $ in wide generality, other than the ad hoc strategies that
Hill-Lawson \cite{HLShimura} and Lawson \cite{LawsonD15} have
employed. One could dream of an approach to connective covers which
integrates duality: assuming that $\Rc$ is Anderson self-dual, without
necessarily a gap in its homotopy, somehow peel off a connective piece
from its coconnective dual, but for the present this is only a fantasy.

\subsection{Description of contents }
We start by giving an account of Anderson duality. The main point of
this is to explain its limitations and to make explicit the way this works under change of ring
spectra. 

We recall the definition of Gorenstein ring spectra and Gorenstein
duality. The  Gorenstein condition only makes sense when we have
a counterpart of a  residue field. However Gorenstein duality makes
sense more generally.   Under
orientability hypotheses the Gorenstein condition implies Gorenstein 
duality.  

It is then straightforward to compare Gorenstein and Anderson duality,
and we illustrate the usefulness of this in a number of cases. 

Finally we finish by describing how to use Molien's theorem to predict
the change of shift under passage to invariants.\footnote{It is
  characteristic that Dave Benson not only wrote the book \cite{Benson} from which
  we learnt this result but also illustrated it for us with numerous
  examples. We are grateful to him for his exposition, his vast range
  of interesting examples, and the delightful process of explanation. }

\subsection{Conventions}
We work in the homotopy category of modules over a ring
spectrum. However we also need to know that there is a ring
spectrum of endomorphisms of a module spectrum. For definiteness, we
work with EKMM-spectra \cite{EKMM}, but our results
are not sensitive to models, so apply in other contexts with a
homotopically meaningful symmetric
monoidal smash product and internal Hom spectra. 
 
Given a spectrum $X$, we write $\pi_*X=X_*$ for its coefficients, and
we note that  if $M$ is an $R$-module 
$$\pi_*(M)=[\bbS , M]_*=[R, M]^R_*, $$
where $\bbS$ is the sphere spectrum, and the superscript $R$ refers to
working in the category of $R$-modules.

The basic context is that we are  given a connective commutative ring
spectrum. It is convenient to use the traditional convention of using
lower case for connective covers, so we write 
$\rc$ for the ring spectrum and   $K=\pi_0(\rc)$. By killing homotopy
groups in commutative ring spectra, we have  a map $\eps: \rc\lra HK$ of
commutative ring spectra. In our main examples $K$ will be an $\bbF_p$ or a localization of $\Z$.

\section{Anderson duals}
The construction of Anderson duals is a two step process. For
injective modules we apply Brown representability (to get the
so-called Brown-Comenetz duals) and then we use cofibre sequences
to obtain Anderson duals for modules of injective dimension 1. Since
we are usually working over a field or a localization of a discrete
valuation ring this covers many cases of interest. 
Unfortunately, the construction cannot be much generalized (at injective
dimension 2 choice is involved, and at higher dimension the
construction  is often obstructed).

\subsection{Construction of Brown-Comenetz duals}
The basis for Anderson duality is that we can uniquely lift injective
coefficient modules to module spectra. In general, we are in a situation where data as below is given.

\begin{itemize}
\item  We have maps of commutative ring spectra
$$\bbS \lra  S \lra R. $$

Often we will take $S$ to equal the sphere spectrum $\bbS  $ or $R$
itself, but it is useful to retain some flexibility.

\item Additionally, we have a map of graded rings
$$A_*\lra R_*.$$

There is no requirement  that $A_*\lra R_*$ is
induced by a map of ring spectra.  For example, we always have the unit map $A_* = \Z \lra R_*$ (in degree zero), and this is what plays a role in classical Brown-Comenetz duality \cite{BrownComenetz}.

The most common and important instance of the above occurs by taking
$K=\pi_0(R)$, and declaring $A_*=K$ in degree zero, i.e.  we consider
the map
$$K \lra R_*.$$
\end{itemize}

The construction is that we take an injective $A_*$-module $J$ and
consider the   functor
$$\xymatrix{
\Rmod \ar[r] &\Abstar\\
X\ar@{|->}[r] & \Hom_{A_*}(\pi_*(X), J).
}$$

Since $J$ is injective, this is a cohomology theory, and by Brown representability there is an $R$-module $J^R=J_{A_*}^R$ so that
$$[X, J^R]^R_*=\Hom_{A_*}(\pi_*(X), J). $$
Slightly more generally, for an $R$-module $M$ we may define $J^M=J^M_{A_*}$ by
the equation 
$$[X, J^M]^R_*=\Hom_{A_*}(\pi_*(X\tensor_R M), J). $$
One quickly checks that 
$$J^M\simeq \Hom_R(M,J^R), $$
and we say $J^M$ is the Brown-Comenetz $J$-dual of $M$. 
 Of course, $J^R$ is itself the Brown-Comenetz $J$-dual of $R$, and this is the case
we will use the most. 

\subsection{Properties of Brown-Comenetz duals}
We highlight four properties of the Brown-Comenetz dual.

(P0) {\em (Homotopy groups)}  By construction, 
$$\pi_*(J_{A_*}^M)=\Hom_{A_*}(\pi_*(M), J). $$

(P1) {\em (Eilenberg-MacLane spectra)} If $R=HK$ is an
Eilenberg-MacLane spectrum and $J$ is an ungraded
injective $K$-module, then
$$J_K^{HK}=HJ. $$

Given a ring map $S\lra R$ and an $S$-module $N$ let us write 
$$\coext{N}{S}{R}=\Hom_S(R, N)$$
for the coextended module. 

(P2) {\em (Coextension of scalars I)} Given $S\lra R$ and $A_*\lra S_*\lra R_*$ we have
$$J_{A_*}^R=\coext{(J_{A_*}^S)}{S}{R}.$$
More generally, if $N$ is an $S$-module, we have 
$$J_{A_*}^{R\tensor_SN}=\coext{(J_{A_*}^N)}{S}{R}.$$
\begin{proof}
For an $R$-module $X$ we have 
$$[X, J^R_{A_*}]^R=\Hom_{A_*}(\pi_*X, J) =[X, J^S_{A_*}]^S =[X,
\coext{(J^S_{A_*})}{S}{R}]^R. $$
\end{proof}

(P3) {\em (Coextension of scalars II)} Given $A_*\lra R_*$ we note
that $\coext{J}{A_*}{R_*}$ is injective and then we have
$$(\coext{J}{A_*}{R_*})_{R_*}^R \simeq J_{A_*}^{R}. $$
\begin{proof}
For an $R$-module $X$ we have 
$$[X, (\coext{J}{A_*}{R_*})_{R_*}^R]^R=
\Hom_{R_*}(\pi_*X, 
\coext{J}{A_*}{R_*}) =\Hom_{A_*}(\pi_*X, J) 
=[X, J^R_{A_*}]^R. $$
\end{proof}

\begin{remark}
 Since coextension is a well-known construction, Property (P2) means
 that  we only ever need the special case of the construction going
from modules over coefficients $A_*$ to modules over an initial ring
spectrum to whose coefficients $A_*$ maps. Property (P3) means that  we
only ever need the special case of the Anderson construction going
from modules over coefficients to modules over the ring spectrum. 
\end{remark}

\subsection{Construction of Anderson duals}
\label{subsec:ConstAnderson}
Now we suppose given an $A_*$-module $L$  of
injective dimension 1 with chosen resolution 
$$0\lra L \lra J_0\lra
J_1\lra 0. $$
We then define  $L^R_{A_*}$ by the fibre sequence 
$$ L_{A_*}^R
\lra (J_0)_{A_*}^R \lra (J_1)_{A_*}^R. $$
We note that the maps are determined by the defining properties and
the original resolution, and it is not hard to show the spectrum  is
independent of the resolution. The classical example \cite{A1} is that
of $L = \Z = A_*$.

As for the Brown-Comenetz case, we may also define the Anderson $L$-dual
of an $R$-module $M$, either by replacing $R$ by $M$ in the above
construction, or directly by taking $L_{A_*}^M=\Hom_R(M,
L_{A_*}^R)$. Again, $L_{A_*}^R$ is itself the Anderson $L$-dual of $R$.

\subsection{Properties of Anderson duals}
\label{subsec:PropAnderson}
The properties of the Anderson dual then follow from those of the
Brown-Comenetz dual.  We suppose that $L$ is an $A_*$-module of
injective dimension $\leq 1$. 

(P0) {\em (Homotopy groups)} There is a natural  exact sequence
$$0\lra \Ext_{A_*}^1(\Sigma \pi_*(R),L) \lra \pi_*(L_{A_*}^R)\lra \Hom_{A_*}(\pi_*(R),
L)\lra 0, $$
and more generally one which computes the homotopy groups of  the dual
of any $R$-module $M$, using $[M,L_{A_*}^R]^R_*=[R, L_{A_*}^M]^R_*$:
$$ 0\lra \Ext_{A_*}^1(\Sigma \pi_*(M),L) \lra [M, L_{A_*}^R]^R_* \lra \Hom_{A_*}(\pi_*(M),
L)\lra 0 .$$

(P1) {\em (Eilenberg-MacLane spectra}) If $R=HK$ is an
Eilenberg-MacLane spectrum and $L$ is an ungraded $K$-module then 
$$L_K^{HK}\simeq HL. $$

(P2) {\em (Coextension of scalars I)} Given $S\lra R$ and $A_*\lra S_*\lra R_*$,  we have
$$L_{A_*}^R\simeq\coext{(L_{A_*}^S)}{S}{R}. $$
More generally, for an $S$-module $N$ we have
$$L_{A_*}^{R\tensor_S N}\simeq\coext{(L_{A_*}^N)}{S}{R}. $$

The main case of interest is that if
$A_*=\Z$ we need only use the classical Anderson dual of the
sphere:
$$\Z^R\simeq \coext{(I_{\Z})}{\bbS}{R}$$
where $I_{\Z}=\Z^{\bbS}$ is the usual Brown-Comenetz dual of the
sphere. Similar comments apply to localizations of $\Z$. 

(P3) {\em (Coextension of scalars II)} Given $A_*\lra R_*$ we note
that we may coextend  the resolution of $L$ to show
$\coext{L}{A_*}{R_*}$ is of injective dimension $\leq 1$ and then we 
have 
$$(\coext{L}{A_*}{R_*})_{R_*}^R \simeq L_{A_*}^R.$$

\section{The Gorenstein condition}
\label{sec:Gor}

We recall the basic language and results of Gorenstein ring spectra
from \cite{DGI}. Because of our applications, we will work with a
map $\rc\lra HK$, and we assume $\rc$ is connective and denote
$K=\pi_0(\rc)$. 

\subsection{Cellularization}
An $\rc$-module $X$ is said to be {\em $HK$-cellular} if it
 is in the localizing subcategory of $HK$ (i.e. it is  constructed from $HK$ using triangles and coproducts). 
An $HK$-cellularization of an $\rc$-module $M$ is a map $X\lra M$ so 
that $X$ is  $HK$-cellular and the map is an $\Hom_{\rc}(HK,
\cdot)$-equivalence. The $HK$-cellularization is unique up to
equivalence of $\rc$-modules and we write $\cell_{HK}M$ for it.

\subsection{Morita theory}
\label{subsec:Morita}
 We say that the $HK$-cellularization of an $\rc$-module $M$ is 
{\em effectively constructible} if the natural evaluation map 
$$\Hom_\rc(HK, M)\tensor_{\cE}HK \lra M$$
is $HK$-cellularization, where $\cE =\Hom_\rc(HK,HK)$. 

We recall that $HK$ is {\em proxy-small} if $HK$ finitely builds a small
object $\HKt$ which generates the same localizing subcategory of $R$-modules. 
The proxy-smallness condition is very mild, but in most of our applications
here we  will be in the situation that $HK$ is actually small so that we may take $\HKt=HK$.

The fact \cite[4.9]{DGI} we use is that if $HK$ is proxy-small as an $\rc$-module, then every $\rc$-module has an
effectively constructible $HK$-cellularization. 

\subsection{The Gorenstein condition}
The basic definition was given for ring spectra in \cite{DGI}. 

\begin{defn}
We say that  $\rc\lra HK$ is {\em Gorenstein of shift $a$} if $\Hom_\rc(HK,\rc)\simeq
\Sigma^aHK$. 
\end{defn}

If we suppose $\rc\lra HK$ is Gorenstein of shift $a$, we may wonder how this compares to other
modules  $\bbI$ lifting $HK$ in the sense that $\Hom_{\rc}(HK,
\Sigma^a \bbI) \simeq \Sigma^aHK$.  For
example the Anderson dual  spectrum $\bbI=K_K^{\rc}=:K^{\rc}$ as in
Subsections \ref{subsec:ConstAnderson} and  \ref{subsec:PropAnderson} qualifies as the `trivial' lift,
and in Section \ref{sec:GorDAndD} and beyond, we will restrict
attention to that case. For now just assume that $\bbI$ is an
$HK$-cellular $\rc$-module with the required lifting property, and note that the notions of Gorenstein orientability and duality below implicitly depend on $\bbI$.

If $\rc\lra HK$ is Gorenstein,  we have  
$$\Hom_\rc(HK,\rc) \simeq \Sigma^aHK \simeq \Hom_\rc(HK, \Sigma^a\bbI). $$
We note that the ring spectrum $\cE =\Hom_\rc(HK,HK)$ acts on the right of
both of these modules. 
 
\begin{defn}
We say that $\rc$ is {\em orientably} Gorenstein if  the equivalence
$$\Hom_\rc(HK,\rc)\simeq \Hom_\rc(HK,\Sigma^a \bbI)$$
is an equivalence of right $\cE$-modules. 
\end{defn}

\subsection{Gorenstein duality}
If $\rc\lra HK$ is orientably Gorenstein and $HK$ is proxy-small, we
may apply the equivalence from Morita theory 
(Subsection \ref{subsec:Morita}) to deduce
$$\cell_{HK}\rc\simeq \Sigma^a \cell_{HK}\bbI .$$

For example the Anderson dual $\bbI=K^{\rc}$ is bounded above, with
each homotopy group a $K$-module, and  hence it is already $HK$-cellular,
so that  $\cell_{HK}K^{\rc}=K^{\rc}$. The above condition translates to an equivalence
$$\cell_{HK}\rc \simeq \Sigma^a K^{\rc}.$$

\begin{defn}
We say that $\rc\lra HK$ has {\em torsion Gorenstein duality} of
shift $a$ if 
$$\cell_{HK}\rc \simeq \Sigma^a \bbI.$$
\end{defn}

Rather often this is used by completing both sides, which is to say
applying the functor
$$(\cdot)_{HK}^{\wedge}=\Hom_\rc(\cell_{HK} \rc, \cdot). $$

\begin{defn}
We say that $\rc\lra HK$ has {\em complete Gorenstein duality}
of shift $a$ if 
$$\rc_{HK}^{\wedge}\simeq \Sigma^a \bbI_{HK}^{\wedge}. $$
\end{defn}

\begin{remark}
\label{rem:conncomp}
Since $\rc$ is connective it is the inverse limit of its Postnikov
sections and hence $HK$-complete, i.e. $\rc_{HK}^{\wedge}\simeq \rc$ and the
condition simplifies to the statement 
$$\rc \simeq \Sigma^a \bbI_{HK}^{\wedge}. $$
\end{remark}

In fact the two Gorenstein duality conditions are equivalent, so that
when no emphasis is necessary we refer simply to `Gorenstein duality'.

\begin{lemma}
The torsion and complete Gorenstein duality statements are
equivalent. 
\end{lemma}

\begin{proof}
Since the map $\cell_{HK}\rc\lra \rc$ is an $HK$-cellular equivalence, it is
clear that the torsion duality implies complete duality by taking
completions, since
$$\Hom_\rc(\cell_{HK} \rc, \cell_{HK} \rc) \simeq \Hom_\rc(\cell_{HK} \rc,  \rc) \simeq \rc_{HK}^{\wedge}.$$

To recover the torsion duality from complete duality, we use
$HK$-cellularizations as follows.

In fact completion is a cellular equivalence
rather generally.  We consider the completion map 
$$M=\Hom_\rc(\rc, M) \lra \Hom_\rc(\cell_{HK} \rc, M)$$
and apply $\Hom_\rc(HK, \cdot)$ to get
$$\Hom_\rc(HK\tensor_\rc \rc, M) \lra \Hom_\rc(HK\tensor_\rc \cell_{HK} \rc, M).  $$
We observe this is an equivalence; indeed, since $HK$ is
$HK$-cellular and $HK$-cellularization is smashing, the map 
$$HK\tensor_\rc \cell_{HK} \rc\lra HK \tensor_\rc \rc$$
is an equivalence. Thus 
$$\cell_{HK} M\simeq \Hom_\rc(\cell_{HK}\rc,M)\tensor_{\rc} \cell_{HK}\rc\simeq \cell_{HK}(M_{HK}^{\wedge})$$
as required.
\end{proof}

\subsection{Gorenstein duality relative to $\Fp$}
We  consider the statement of Gorenstein duality for $\rc\lra 
H\Fp$ when $K=\pi_0(\rc) \cong \Z$ (or equally when $K=\Z_{(p)}, \Z_p^{\wedge}$),  referring to the discussion in the previous subsection  for comparison. 

As before we start by assuming   $\rc\lra H\Fp$ is Gorenstein of shift $a$, and note that this gives an equivalence 
$$\Hom_\rc(H\Fp, \rc) \simeq \Sigma^{a}H\Fp \simeq \Hom_\rc(H\Fp, \Sigma^{a+1} K^\rc). $$
The difference from the case relative to $HK$ is that 
$$\cell_{H\Fp}(K^\rc) \simeq \Sigma^{-1}(\Z/p^{\infty})^\rc.  $$
The appropriate definition is then clear. 

\begin{defn}
When $K=\pi_0(\rc)=\Z, \Z_{(p)}, \Z_p^{\wedge}$, we say that $\rc\lra H\Fp$ has Gorenstein duality of shift $a$ if 
$$\cell_{H\Fp}\rc\simeq \Sigma^a (\Z/p^{\infty})^\rc. $$
\end{defn}
As before, if  $\rc\lra H\Fp$ is proxy regular and there is a unique
action of $\Hom_\rc(\Fp, \Fp)$ on $\Fp$, 
then Gorenstein implies Gorenstein duality. 

In the context where both make sense, we show in the next subsection
that this Gorenstein duality is equivalent to the duality of Mahowald-Rezk \cite{MahowaldRezkdual}.

\subsection{Mahowald-Rezk duality}
\label{subsec:MR}
\newcommand{\type}{\mathrm{type}}
\newcommand{\fptype}{\mathrm{fp\!\!-\!\!type}}

Mahowald and Rezk \cite{MahowaldRezkdual} consider the class of  {\em fp-spectra} (connective, $p$-complete and whose mod $p$ 
homology is a finitely presented comodule over the Steenrod 
algebra). The {\em type} of a $p$-local finite complex $F$ is defined by
$$\type(F)=\min\{n \st K(n)_*F\neq 0\}, $$
where $K(n)$ is the $n$th Morava $K$-theory at $p$. 
 The  {\em fp-type} of an   fp-spectrum $X$ is defined by 
$$\fptype(X)=\min \{\type(F)-1 \st F \mbox{ is a finite complex and
} \pi_*(X\sm F) \mbox{ is a finite group }\}.$$
 For example,
$ko$ and $ku$ are  fp-spectra  of fp-type 1, and $tmf$ is an fp-spectrum of fp-type 2. 

 If $\rc$ is a ring spectrum of fp-type $n$, such that its mod-p homology is self-dual in a suitable sense, then Mahowald and Rezk show that there is a duality equivalence
$$(\Z/p^{\infty})^{C_n^f\rc}\simeq \Sigma^c\rc .$$
This is satisfied in a number of cases, including $ko$, $ku$, and
$tmf$ \cite[Proposition 9.2, Corollary 9.3]{MahowaldRezkdual}.
Here  $C_n^f$ is the $n$th finite chromatic cellularization (i.e., the cellularization with
respect to a finite type $n+1$ complex $F$).  A more specific 
construction proceeds by constructing  a cofinal inverse system  of
generalized Moore spectra $S^0/I=
S^0/v_0^{i_0}, v_1^{i_1}, \ldots , v_n^{i_n}$ and then taking 
$$C_n^fX=\hocolim_I F(S^0/I, X). $$

\newcommand{\cellFpmod}{\cell_{H\Fp}(\mbox{$\rc$-mod})}
\newcommand{\cellFspectra}{\cell_{F}(\mbox{Spectra})}
\newcommand{\cellFR}{\cell_{F\wedge \rc}}

\begin{lemma}
If $\rc$ is an fp-spectrum of fp-type $n$ then there is a
natural equivalence $C_n^fM\simeq \cell_{H\Fp}M$ for $\rc$-modules $M$.
\end{lemma}

\begin{proof}
The proof consists of two steps: identify $C_n^f M$ with the cellularisation in $\rc$-modules $\cellFR$, and then show that the localising subcategories $\langle H\Fp \rangle$ and $\langle F\wedge \rc \rangle$ of $\rc$-modules, generated by $H\Fp$ and $F\wedge \rc$, respectively, are equal.

For the first step, we check that $\cellFR M$ has the required universal property. Let $M[1/F\wedge \rc]$ denote the cofibre of the natural map $\cellFR M \to M$; then spectrum maps from $F$ to $M[1/F\wedge \rc]$ are the same thing as $\rc$-module maps from $F\wedge \rc $ to it, but by construction those are all null. Next, we need to know that the spectrum underlying $\cellFR M $ is in the localising subcategory of spectra generated by $F$. Since colimits commute with smash product, this follows since the $\rc$-module $\cellFR M $ is in the localising subcategory of $\rc$-modules generated by $F \wedge \rc$.

For the second step, the key property is that $F \wedge \rc $ is a
finite wedge of copies of $H\Fp$ by  \cite[Proposition
3.2]{MahowaldRezkdual}. Hence, $F \wedge \rc$ is in the localising
subcategory $\langle H\Fp \rangle$ (argue by induction that if
$\pi_*(M)$ is a finite dimensional vector space it is finitely built
by $H\Fp$; for the inductive step, if $M$ has bottom homotopy in
degree 0, killing homotopy groups in $\rc$-modules,  gives a map $M \lra
H\Fp$ non-zero in $\pi_0$). Conversely $H\Fp$ is in  $\langle
F\wedge \rc \rangle$ (if $M$ is a module which is a finite wedge of
copies of $H\Fp$ as a spectrum, then we can construct a map from a finite
wedge of copies of $F\sm \rc$ that is surjective on the bottom
homotopy; since $F\sm \rc$ is small, repeating this and passing to
direct limits,  we may kill all homotopy.  To construct the map, note
that for any chosen element of $\pi_0$ there is a map $F\lra M$ which 
maps onto it, and we extend it to an $\rc$-module map 
$\rc\sm F\lra M$). 
\end{proof}

Accordingly, the Mahowald-Rezk duality statement reads 
$$(\Z/p^{\infty})^{\cell_{H\Fp}\rc}\simeq \Sigma^c\rc. $$
Assuming the homotopy groups of $\rc$ are profinitely complete,  we may dualize to find 
$$\cell_{H\Fp}\rc \simeq \Sigma^{-c}(\Z/p^{\infty})^\rc$$

When $\pi_0(\rc)=\Z_p^{\wedge}$, this is precisely the
statement that  $\rc\lra H\Fp$ is Gorenstein of shift $-c$. Summarising, the above gives the following conclusion.

\begin{lemma}
If $\rc$ is an fp-spectrum of fp-type $n$, whose homotopy groups are $p$-complete,  then $\rc \to H\Fp $ is Gorenstein of shift $-c$ if and only if $\rc$ is Mahowald-Rezk self-dual of shift $c$.\qqed
\end{lemma}

\section{Gorenstein duality and Anderson self-duality}
\label{sec:GorDAndD}
In this section we explain that Gorenstein duality for a connective
ring spectrum gives an Anderson self-duality for the
associated non-connective spectrum. 

We note that  Anderson duality exchanges connective and coconnective
spectra, so we cannot expect to have self-duality for connective
spectra. Similarly, periodic spectra often fail to have residue
fields, so that the Gorenstein condition usually makes no sense for
them. Accordingly each approach plays an  essential role.

\subsection{Nullifying $HK$}
From our connective ring spectrum $\rc$ and the map $\rc \lra HK$
obtained by killing higher homotopy groups we may form a cofibre sequence
$$\cell_{HK}\rc \lra \rc \lra \rc [1/HK]$$
where $\rc \lra \rc [1/HK]$ is the initial map to a spectrum with no
maps from $HK$. We take $\Rc =\rc [1/HK]$, and it is a commutative
ring spectrum since $\rc$ is.

\subsection{Anderson self-duality from Gorenstein duality}
We are ready to bring the threads together. The most interesting 
implication is that Anderson self-duality follows from Gorenstein
duality. 

\newcommand{\NKr}{\rc [1/HK]}
\begin{prop}
\label{prop:GorDAndD}
If  $\rc\lra HK$ has Gorenstein duality  of shift $a$ then $\NKr$ 
has Anderson self-duality with shift $a+1$ in the sense that
 $$K^{\NKr}\simeq \Sigma^{-a-1}\NKr.$$ 

Furthermore, 

(i) $K^\rc \simeq \Sigma^{-a}\cell_{HK}\rc$.

(ii) 
The map $\eps: \cell_{HK} \rc\lra   \rc$ is self dual: if we apply  $\Hom_\rc(\cdot ,
K^\rc)$ to $\eps$,  we obtain the $a$th desuspension of $\eps$. 
\end{prop}

\begin{remark}
Note that the Anderson self-duality statement
makes it natural to write the suspension on the side of the ring
$$K^{\Rc}\simeq \Sigma^{-a-1}\Rc, $$
since it says the Anderson dual is a shift of the original ring. 
The Gorenstein duality statement makes it natural to put the suspension
on the side of the  Anderson dual
$$\cell_{HK}\rc \simeq \Sigma^a K^{\rc}, $$
since it says the cellularization is a shift of a naive expectation. 

Of course it is easy to pass between the two, but the first convention
focuses on a shift (viz $-a-1$) that is usually positive whereas the second convention 
focuses on a shift (viz $a$) that is usually negative. 

\end{remark}

\begin{proof}
Part (i) is a restatement of Gorenstein duality, and
the Anderson self-duality  is an immediate consequence of Part (ii).

It remains only to prove that the map in (ii) is self dual. However we
note that  maps $\eps: \cell_{HK} \rc\lra \rc$ are easily classified since
$\Hom_\rc(\cell_{HK} \rc , \rc)\simeq \rc$ with $\pi_0(\rc)=K$. 

To see that the dual of $\eps$ is as required, let 
$$\rho : \cell_{HK} \rc \stackrel{\simeq}\lra \Sigma^a K^\rc$$
be the given equivalence. Since $\rho$ is an equivalence
we may use $\Hom_\rc( \cdot, \Sigma^{-a}\cell_{HK} \rc)$ as the dualization. Then 
$\eps$ dualizes to 
\begin{align*} \eps^*: \Sigma^{-a}\cell_{HK} \rc\simeq & \Hom_\rc(\rc, \Sigma^{-a}\cell_{HK} \rc) \lra \\
&\lra \Hom_\rc(\cell_{HK} \rc,
\Sigma^{-a}\cell_{HK} \rc)\simeq \Sigma^{-a}\rc_{HK}^{\wedge}\simeq
\Sigma^{-a}\rc, 
\end{align*}
where the last equivalence is because $\rc$ is connective (see Remark \ref{rem:conncomp}).
It is easy to see this has the universal property of cellularization
and is therefore the suspension of $\eps$.
\end{proof}

On the other hand, if we have Anderson self-duality in the sense that
 $$\Sigma^{a+1} K^{\NKr}\simeq \NKr,$$ then it is not
automatic that $\rc$ has Gorenstein duality without additional
 connectivity statements (for example Meier \cite{MeierLevels} shows $Tmf_1(23)$ is
 Anderson self-dual, with $a=0$, whereas one can see by considering complex
 modular forms with level 23 structure that its connective cover does not
enjoy  Gorenstein duality). 

\begin{lemma}
\label{lem:AndDGorD}
Suppose that 
$$K^{\NKr}\simeq \Sigma^{-a-1}\NKr $$
with $a\leq -2$. 

If  $\pi_i(\cell_{HK} \rc)=0$ for $i\geq a+1$, and $\pi_{a}(\cell_{HK} \rc)$ is projective over 
$K$, then $\rc$ has  Gorenstein duality of
shift $a$. 
\end{lemma}

\begin{proof}
We apply $\Hom (\cdot, K^\rc)$ to  the cofibre sequence 
\begin{align}\label{CechCofSeq}
\cell_{HK} \rc \lra \rc \lra \NKr
\end{align}
 to  obtain 
$$\Hom_\rc(\cell_{HK} \rc, K^\rc)\lla K^\rc\lla \Sigma^{-a-1}\NKr. $$
Suspending $a$ times and taking mapping cones, we obtain the cofibre sequence
\begin{align}\label{CechDualSeq}
\NKr \lla \Sigma^a\Hom_\rc(\cell_{HK} \rc , K^\rc) \lla \Sigma^a K^\rc,
\end{align}
and we want to check that this is equivalent to the original \eqref{CechCofSeq}.

From the hypotheses, $\pi_t(\Sigma^a \Hom_\rc(\cell_{HK} \rc, K^\rc))$ is zero for $t\leq -1$. Indeed, from the Anderson dual Property (P0), this homotopy group sits in an exact sequence 
\begin{align*}
0 \lra \Ext^1_K(\pi_{-t+a-1}\cell_{HK} \rc, K ) &\lra \pi_t(\Sigma^a \Hom_\rc(\cell_{HK} \rc, K^\rc)) \\
&\lra \Hom_{K}(\pi_{-t+a}(\cell_{HK} \rc), K) \lra 0,
\end{align*}
and for $t \leq -1$, both the $\Hom$ and $\Ext$ term vanish.
Hence
$$\rc\simeq (\NKr)[a+2, \infty)\simeq \left( \Sigma^a \Hom_\rc(\cell_{HK} \rc,
  K^\rc)\right) [a+2, \infty) \simeq \Sigma^a \Hom_\rc(\cell_{HK} \rc, K^\rc);$$
  the first and second equivalence are because $(\cell_{HK}\rc)
  [a+2,\infty)$ and $(\Sigma^a K^{\rc}) [a+2,\infty) $ respectively are contractible.
  
 Thus the middle term of the sequence \eqref{CechDualSeq} is $\rc$; it remains to check that its map to $\rc[1/HK]$ satisfies the requisite universal property. This follows since the fibre $\Sigma^a K^\rc$ is clearly $HK$-cellular. We conclude that $\cell_{HK} \rc \simeq \Sigma^a K^\rc$ as required.
\end{proof}

\section{Examples with polynomial or hypersurface coefficient rings }
\label{sec:hypgeom}

There are quite a number of examples that are algebraically very
simple, so that we can apply our results without additional work, 
and we discuss a selection of those  here. 

\subsection{The \Cech\ complex}
When the coefficient ring is  simple, we have very algebraic models
of the cellularization $\cell_{HK} M$ and $M[1/HK]$. We briefly recall
the construction here (see \cite{GMjames} for more details). 

Suppose that $\fn =(x_1, \ldots , x_r)$ is an ideal in the coefficient
ring $\rc_*$. 
There is an elementary construction of the \Cech\ spectrum $\Cn \rc$ as follows. 
First we form the stable Koszul complex 
$$\Gamma_{\fn} \rc = \Gamma_{(x_1)}\rc \tensor_\rc \cdots \tensor_\rc \Gamma_{(x_r)}\rc $$
where $\Gamma_{(x)}\rc=\fibre(\rc\lra \rc[1/x])$. We note that the homotopy type does not
depend on the particular generators $x_i$. Indeed, it is obvious that replacing generators $x_i$ by powers has no effect,
and it is not hard  to see that $\Gamma_{\fn} $ only depends on the radical of
the ideal $\fn$. 

Now define $\Cn \rc$ by the fibre sequence 
$$\Gamma_{\fn} \rc \lra \rc \lra \Cn \rc .$$
It is easy to check that $\Cn \rc$ is a commutative ring up to homotopy,
but  it can also be constructed \cite{GMjames}  as the nullification 
$$\Cn \rc \simeq \rc [\frac{1}{(\rc/\bfx)}],$$ where 
$$\rc/\bfx =\rc/x_1\tensor_{\rc}\cdots \tensor_{\rc}\rc/x_n$$ 
is the unstable Koszul complex. It follows that $\Cn \rc$ admits the structure of a  commutative ring.

The case we have in mind is that $\rc$ is connective with $\rc_*$ Noetherian and
$$\fn=\ker (\rc_*\lra \rc_0=K). $$
The relevance is clear from a lemma. 

\begin{lemma}
Suppose $\rc_*$ is a polynomial ring over $K$ or a hypersurface (i.e. $\rc_*=K[x_1, \ldots , x_n]/(f)$
with  $f$ of positive degree).
The $\rc$-module $HK$ is proxy-small. For an $\rc$-module $M$ we have equivalences
$$\cell_{HK}M\simeq \Gn M \mbox{   and } \Cn M \simeq M[\frac{1}{HK}] .$$
\end{lemma}

\begin{proof}
We will show that  $HK$ finitely
builds $\HKt= \rc/\bfx $ and $\HKt$ builds $HK$. This shows that $\HKt$
is a witness for the proxy-smallness of $HK$ and in particular shows
that  $HK $ and $\HKt$ generate the same localizing subcategory. 


If $\rc_*$ is a polynomial ring then $HK$ is itself small: we 
take $\HKt=\rc/\bfx $. We have a map $\rc \lra \HKt$ and
calculation immediately shows it is an isomorphism in $\pi_0$ so that
$HK\simeq \HKt$.

If $\rc_*=K[x_1, \ldots , x_r]/(f)$ with $f$ of degree $s$ then we
take $\HKt =\rc/\bfx$ and calculate
$\pi_*(\HKt)=K[\phi]/(\phi^2)$, where $\phi$ is of degree $s+1$. We
need only observe this is additively the homology of the short cochain
complex 
$$\Sigma^s
\rc_*\stackrel{f} \lra \rc_*.$$
 To  see this, consider the polynomial 
ring  $P= K[x_1, \ldots , x_r]$ and form the Koszul complex $KP$  for the elements
$f, x_1, \cdots , x_r$. If we view $KP$ as a multicomplex and  take homology  in the order stated, it
is the homology of the displayed complex.  If we take homology in the order
$ x_1, \ldots, x_r, f$ then it is evidently $K[\phi]/(\phi^2)$. 
Killing homotopy groups in
$\rc$-modules gives a cofibre sequence $\Sigma^{s+1}HK \lra \HKt\lra
HK$ showing that $HK$ finitely builds $\HKt$.

Similarly we may construct $HK$ from $\HKt$ by a process of killing homotopy 
groups, but now using $\HKt$ only. More precisely, we take 
$HK^0=\HKt$ and iteratively construct $HK^{t+1}$ using a cofibre
sequence
$$\Sigma^{t(s+2)-1} \HKt \lra HK^t\lra HK^{t+1}$$
where $\pi_*(HK^t)$ is zero except in degrees 0 and $t(s+2)-1$ where
it is $K$. The attaching map is chosen to be an isomorphism in degree
$t(s+2)-1$. We see that $HK^{\infty}=\hocolim_t HK^t$ is an
Eilenberg-MacLane spectrum, and the map 
$$\rc \lra \HKt=HK^0\lra \hocolim_tHK^t$$ 
is an isomorphism in $\pi_0$ showing that $HK\simeq \hocolim_tHK^t$ as
$\rc$-modules. 
\end{proof}

\subsection{The algebraic context}
As usual we have a connective ring
spectrum $\rc$ with $\pi_0(\rc)=K$.  We assume that $K$ is  a localization or a completion of
a number ring (usually $\Z$), that $\rc_*$ is in even degrees,  free over $K$ and of Krull
dimension 2 and is either polynomial or a hypersurface ring.
Some examples are tabulated in Subsection \ref{subsec:egtab} below. 

In fact we suppose
$$\rc_*=K[x,y,z]/(f) \mbox{ with } |x|=i, |y|=j, |z|=k, |f|=d. $$
The case of a polynomial ring is a little simpler, but in any case it
is covered by taking $f=z$.

Thus $\rc_*$ is a relative complete intersection, and $\rc_*$ is
Gorenstein (and accordingly $\rc$ is itself Gorenstein). Indeed, it is
easy to calculate local cohomology (the cohomology of the stable
Koszul complex), directly or by local duality to see
$$H^*_{\fn}(\rc_*)= H^2_{\fn}(\rc_*)=\Sigma^{a+2} \rc_*^{\vee}$$
where $\rc_*^{\vee}=\Hom_{K}(\rc_*, K)$ and $a=d-(i+j+k)-2.$
Since this is in a single cohomological degree the spectral sequence
\cite[Theorem 4.1]{KEG} for calculating homotopy collapses to give 
$$\pi_*(\Gn \rc)=\Sigma^a\rc_*^{\vee} =\pi_*(\Sigma^aK^\rc).$$
Assuming $a\leq -2$,  the map $\Gn \rc \lra \rc$ is zero in homotopy and 
the cofibre sequence
$$\Gn \rc \lra \rc \lra \Cn \rc$$
gives an isomorphism 
$$\pi_*(\Cn \rc)=\rc_*\oplus \Sigma^{a+1} \rc_*^{\vee} ;$$
since $a$ is even, and $\rc_*$ is in even degrees,  this is an isomorphism of  $\rc_*$-modules. 

From the  algebraic isomorphism
$$\pi_*(\Gn \rc)=\pi_*(\Sigma^aK^\rc), $$
we choose an isomorphism $\pi_a(\Gn \rc)\stackrel{\cong}\lra
\pi_a(\Sigma^aK^{\rc})$, and since the homotopy of $\Gn \rc$ is free
over $K$, the defining property of the Anderson dual gives  a residue map 
$$\rho: \Gn \rc \lra \Sigma^a K^\rc.$$
 To see that $\rho$ is an
equivalence we note that both domain and codomain are $HK$-cellular,
and hence it is enough to show it induces an equivalence
$$\Sigma^a HK=\Hom_\rc(HK, \Gn \rc)\stackrel{\rho_*}\lra \Hom_\rc(HK, \Sigma^{a}
K^\rc)=\Sigma^a HK.$$
We note that this shows that $\rc$ has  Gorenstein duality, since the
spectrum $HK$ has a unique $\cE$-module structure by connectivity.

\subsection{A family of examples}
\newcommand{\fMb}{\overline{\mathfrak{M}}}
Our examples come from derived algebraic geometry. We concentrate on
the case of topological modular forms with level structure for
definiteness. We begin with the compactified moduli stack $\fMb
=\fMb_{ell}(\Gamma)$ of elliptic curves with level $\Gamma$ structure,
on which we have the Goerss-Hopkins-Miller sheaf $\cO^{top}$  of
$E_{\infty}$-ring spectra (see \cite{HLLevel} for the log-\'etale
refinement appropriate for level structures), and then take
$$Tmf(\Gamma) =\Gamma (\fMb_{ell}(\Gamma), \cO^{top}). $$
The homotopy groups of this are calculated through a spectral sequence
$$H^s(\fMb_{ell}(\Gamma); \omega^{\tensor t})\Rightarrow
Tmf(\Gamma)_{2t-s}, $$
where $\omega$ denotes the sheaf of invariant differentials on $\fMb_{ell}(\Gamma)$.
Consider those level structures for which $\fMb_{ell}(\Gamma)$ is representable. Then $\fMb_{ell}(\Gamma)$ 
 is in fact a curve, and thus the spectral sequence collapses to give
$$Tmf(\Gamma)_{2t}=H^0(\fMb_{ell}(\Gamma); \omega^{\tensor t}) $$
and 
$$ Tmf(\Gamma)_{2t-1} = H^1(\fMb_{ell}(\Gamma); \omega^{\tensor t}). $$

Assume that $H^1(\fMb_{ell}(\Gamma); \omega^{\tensor t}) $ is zero for all $t\geq 0$ (which happens in many cases), so that the contribution from $H^1$ is entirely in negative
degrees. Then we may take 
$tmf(\Gamma)$ to be the connective cover of $Tmf(\Gamma)$ and obtain 
$$tmf(\Gamma)_*=H^0(\fMb_{ell}(\Gamma); \omega^{\tensor */2}). $$
However if $H^1$ does not have the vanishing property, it may be much trickier to construct
$tmf(\Gamma)$ with this property. In specific examples, it could be done by hand, by killing the extra homotopy groups of the connective cover, as Hill-Lawson \cite{HLShimura} and Lawson \cite{LawsonD15} do for the similarly behaved topological automorphic forms of discriminants 6 and 15.
However we come by it, we
assume the existence of a spectrum $tmf(\Gamma)$ realizing the $H^0$
part, and a map
$tmf(\Gamma)\lra Tmf(\Gamma)$ inducing a monomorphism on homotopy
groups. 

We continue taking $\fn$ to be the ideal of positive degree elements
of $\rc_*$,
and in our cases this is a finitely generated ideal so that we can
make the localization $tmf (\Gamma)\lra \Cn tmf(\Gamma)$.
\begin{lemma}
\label{lem:Cn}
The map $\ell: tmf(\Gamma )\lra Tmf(\Gamma)$ induces an equivalence
$$\Cn tmf(\Gamma )\simeq Tmf(\Gamma). $$ 
\end{lemma}

\begin{proof}
For brevity, let $\bft=tmf(\Gamma)$, and $\bfT=Tmf(\Gamma)$; we show that $\ell: \bft \lra \bfT$ has the  universal property that $\bft \to \Cn \bft$ enjoys.

First, note that if $x \in \fn$, then $\ell$ induces
$$\bfT[1/x]\simeq \bft[1/x], $$
since the fibre of $\ell$ is bounded
above. 

Let $\bft/\bfx$ be the unstable Koszul complex for some set $\bfx$ of radical generators of $\fn$. It remains to show that 
$$[\bft/\bfx, \bfT]_*^{\bft}=0.$$
For this we note that $\fMb=\fMb_{ell} (\Gamma)$ has a finite open cover by substacks
$\fMb [1/y]$ for $y\in \fn$, and the intersections of these are of the
same form. (For example, we can pull back the cover of $\fMb_{ell}$ by the opens where the modular forms $c_4$ and $\Delta^{24}$ are respectively invertible.) Furthermore, 
$$\bfT[1/y]\simeq \Gamma (\fMb [1/y]; \cO^{top}). $$
By our assumptions,
$$H^0(\fMb [1/y]; \omega^{\tensor
  */2})=\bft_*[1/y]. $$
Since $y$ acts nilpotently on $\bft/\bfx$, we see that 
$$[\bft/\bfx, \bfT[1/y]]_*^{\bft}=0.$$
Since $\fMb$ has a cover whose sets and intersections are all of the
form $\fMb[1/y]$ it follows that $\bfT$ is built from the
spectra $\bfT[1/y]$, and hence
$$[\bft/\bfx,
\bfT ]_*^{\bft}=0$$ 
as required.
\end{proof}

\subsection{Tabulation of examples}
\label{subsec:egtab}
It is helpful to tabulate  a range of examples we can deal with by
these elementary means (i.e. where the coefficient ring is polynomial
or  a complete interesection). 

The first entry in the row is the common name for the ring spectrum,
either topological modular  forms with a level $\Gamma$ structure
$tmf(\Gamma)$ (general reference \cite{HLLevel}) 
or a particular ring of topological automorphic forms with additional
data (general reference \cite{BL}). Each row is a $p$-local or $p$-complete
statement, where $p$ is the second entry.   
The third column gives a finite group of automorphisms of $\rc$. The
homotopy fixed point spectrum will usually have much more complicated
homotopy groups, which may be studied by descent. The degrees of generators are self-explanatory and $a $ is the
 Gorenstein shift.

\newcommand{\tafvi}{taf_{\delta 6}}
\newcommand{\tafviALa}{taf_{\delta 6}^{AL\alpha}}
\newcommand{\tafviALb}{taf_{\delta 6}^{AL\beta}}
\newcommand{\tafviALab}{taf_{\delta 6}^{AL\alpha\beta}}

\newcommand{\Tafvi}{Taf_{\delta 6}}
\newcommand{\TafviALa}{Taf_{\delta 6}^{AL\alpha}}
\newcommand{\TafviALb}{Taf_{\delta 6}^{AL\beta}}
\newcommand{\TafviALab}{Taf_{\delta 6}^{AL\alpha\beta}}

\newcommand{\tafxiv}{taf_{\delta 14}}
\newcommand{\Tafxiv}{Taf_{\delta 14}}

\newcommand{\tafx}{taf_{\delta 10}}
\newcommand{\Tafx}{Taf_{\delta 10}}
\newcommand{\tafxsqrt}{taf_{\delta 10, \sqrt{2}}}
\newcommand{\Tafxsqrt}{Taf_{\delta 10, \sqrt{2}}}

\newcommand{\tafxv}{taf_{\delta 15}}
\newcommand{\Tafxv}{Taf_{\delta 15}}

\medskip
\begin{tabular}{lcc|rrr|r||r|}
$\rc$&$p$&&$\deg (x)$&$\deg (y)$&$\deg (z)$&$\deg (f)$&$a$\\
\hline
$tmf(3)$ &$2$&$BT_{48}$&$2$&$2$&-&-&$-6$\\
$tmf_1(3)$&$2$&$C_2$&$2$&$6$&-&-&$-10$\\
$tmf(2)$&$3$&$\Sigma_3$&$4$&$4$&-&-&$-10$\\
$tmf_0(2)$&$3$&&$4$&$8$&-&-&$-14$\\
\hline
$\tafvi$& 5&two $C_2$  &$8$ &$12$& $24$&$48$&$2$\\
$\tafviALa$&5  & &$8$ &$24$& $24$&$48$&$-10$\\
$\tafviALb$& 5  &&$8$ &$12$& - &- &$-22$\\
$\tafvi$& $ \pm 1$ mod 24&$C_2\times C_2$&  $8$ &$12$& $24$&$48$&$2$\\
$\tafviALab$& $\pm 1$ mod 24 && $16$ &$24$& $44$&$88$&$2$\\
\hline
$\tafxiv$& $3$&&$4$ &$16$& - &- &$-22$\\
\hline
$\tafxsqrt$& $3$&$C_3$&$4$ &$4$& $12$ &$24$ &$2$\\
\hline
$\tafxv$& $2$&$C_8\times C_2$ &$2$ &$6$& $12$ &$24$ &$2$\\
\hline
\end{tabular}

\vspace{2ex}

Although the general features are covered above, we make four cases
explicit. The details of the first set of examples (topological modular forms) can be found in a number of sources, including \cite{MahowaldRezk, Stojanoska2, Stojanoska3}. The second set (topological automorphic forms of discriminant 6), which we summarize below, is based on the work of Hill-Lawson \cite{HLShimura}, as are the next two, topological automorphic forms of discriminants 14 and 10. The last charted example, topological automorphic forms of discriminant 15, is the subject of Lawson's paper \cite{LawsonD15}.

\begin{example}
We consider the spectrum $\rc=\tafvi$ which is the connective version of
the spectrum $\Tafvi =\Cn \rc$ of topological automorphic forms of discriminant
6 \cite[Section 3]{HLShimura}. Note that $a=2$ in this case, but as is done in \cite{HLShimura}, one can still construct a good connective spectrum $\rc$ such that the analogue of Lemma
\ref{lem:Cn} holds. 

The coefficients are 
$$\rc_*=(\tafvi)_*=K[x,y,z]/(f), \mbox{ where } f=3x^6 + y^4 + 3 z^2,  $$
with  
$$K=\Z [1/6], |x|=8, |y|=12, |z|=24 \mbox{ and } |f|=48.$$
Thus $\rc_*$ is a relative complete intersection, and $\rc_*$ is
Gorenstein. Indeed, it is easy to calculate local cohomology, directly or by local duality to see
$$H^*_{\fn}(\rc_*)= H^2_{\fn}(\rc_*)=\Sigma^a \rc_*^{\vee}$$
where $\rc_*^{\vee}=\Hom_{K}(\rc_*, K)$ and $a=48-(8+12+24)-2=2.$
Since this is in a single cohomological degree we have
$$\pi_*(\Gn \rc)=\Sigma^2\rc_*^{\vee},$$
and then the cofibre sequence
$$\Gn \rc \lra \rc \lra \Cn \rc$$
gives
$$\pi_*(\Cn \rc)=\rc_*\oplus \Sigma^3 \rc_*^{\vee},  $$
where the splitting follows by degree and parity. 
\end{example}

\begin{example}
Considering the completion at $p=5$, there are two distinct lifts of the Atkin-Lehner involution $w_6$ on
$\tafvi$, as in \cite[Example 3.12]{HLShimura}; for brevity, we call them $\alpha$ and $\beta$.

(1) The $\alpha$-involution negates $y$, so that $Y=y^2$ is
invariant; $x$ and $z$ are fixed. We take 
$$\rc:=\tafviALa:=(\tafvi)^{hC_2,\alpha}.$$ 
Since 2 is invertible,  the invariants give the homotopy
$$\rc_*=(\tafviALa)_* = K [x,Y,z]/(f) \mbox{ where } f=3x^6+Y^2+3z^2,$$
$$K=\Z_5^{\wedge}, |x|=8, |Y|=24, |z|=24 \mbox{ and } |f|=48.$$
This ring is Gorenstein of shift $a=48-(8+24+24)-2=-10$.

(2) The $\beta$-involution negates $z$, leaving $x$ and $y$ fixed.
We take 
$$\rc:=\tafviALb:=(\tafvi)^{hC_2,\beta}.$$ 
Again, since 2 is invertible,  the invariants give the homotopy
$$\rc_*=(\tafviALb)_* = K [x,y],$$
where
$$K=\Z_5^{\wedge}, |x|=8 \mbox{ and } |y|=12. $$
The ring $\rc_*$, and hence also $\rc$ is Gorenstein of shift $a=-(8+12)-2=-22$.

The difference in shifts in (1) and (2) illustrates that the change in
shift on descent depends on the action. 
\end{example}

\begin{example}
Completing at a prime $p\equiv \pm 1$ mod 24, there are two commuting Atkin-Lehner involutions on
$\tafvi$ and we may take $C_2\times C_2$ invariants, as in \cite[Example 3.11]{HLShimura}. We find that
$X=x^2$, $Y=y^2$ and $T=xyz$ are invariant.  
We take 
$$\rc:=\tafviALab:=(\tafvi)^{h(C_2\times C_2)}.$$ 
Again, since 2 is invertible,  the  invariants give the homotopy
$$\rc_*=(\tafviALab)_* = K [X,Y,T]/(g) \mbox{ where } g=x^2y^2f=3X^4Y+XY^3+3T^2,$$
$$K=\Z_p^{\wedge}, |X|=16, |Y|=24, |T|=44 \mbox{ and } |g|=88.$$
Thus $\rc_*$ and hence also $\rc$ itself is Gorenstein of shift $a=88-(16+24+44)-2=2$.
\end{example}

\section{Invariant theory and descent}
\label{sec:ratdescent}
We imagine given a connective ring spectrum $\rc$ and $\Rc=\Cn \rc$,
and  that a finite group $G$ acts on $\rc$ and hence on $\Rc$. In fact
we suppose that $\rc$ and $\Rc$ are $G$-spectra, but we will only make
use of the naive action. 
In the examples we know, $\Rc^G\lra \Rc$ is a Galois extension, so that in
particular $\Rc^G\simeq \Rc^{hG}$, and $\rc^G$ is the connective cover
of $\Rc^G$.

We assume that it has been proved that the ring spectrum $\rc$ has Gorenstein
duality and $\Rc$ is Anderson self-dual, and we are interested in
proving the good properties descend to  $\rc^G$ and $\Rc^G$.
 It is well known in algebra that the Gorenstein property need not descend to
rings of invariants, and that when it does, there will usually be a
change (the Solomon Supplement) in the Gorenstein shift. 

Since rationalization commutes with taking invariants, one can learn
about the general question of descent by considering the rational
case, which is essentially algebraic. In particular, if $\rc^G$ is
Gorenstein with Solomon Supplement $b$ then this will also be true 
rationally, so we obtain a necessary condition for $\rc^G$ to be
Gorenstein and a prediction of its Gorenstein shift.

The purpose of this section is to describe what happens in the
algebraic case, and to note that the Solomon Supplement is predicted
from Solomon's Theorem in invariant theory, and can be calculated from the character of
the action of $G$ on the polynomial generators.

Since the \Cech\ and  homotopy fixed point constructions commute with
localization we assume for the remainder of this section that $\rc_*$ is
rational. 

We thank D.J.Benson for his illustrated tutorials and 
we recommend \cite{Benson} for an account of the relevant
invariant theory.

 \subsection{Context}
In invariant theory, the best understood case is that of a polynomial
ring, so let us assume $\rc_*$ is a polynomial ring 
$$\rc_*=K[x_1,\ldots , x_r],$$
in $r>0$ variables, where $x_i$ is of even degree $d_i>0$ and $K$ is a $\Q$-algebra. 
Of course  $\rc_*$ (and hence $\rc$ itself) is Gorenstein of shift 
$$  a=-(d_1+\cdots+d_r)-r . $$
Since  the degrees $d_i$ are positive,  $a\leq -2$,  and we have (additively) 
$$  \Cn (\rc)_*=\rc_* \oplus \Sigma^{a+1} \rc_*^{\vee}$$
giving an Anderson self-duality shift of $a+1$. 

It is convenient to collect together the polynomial generators and say
$$\rc_*=K[V]$$
where $V=Q(\rc_*)$ is the indecomposable quotient, a graded free
$K$-module. 

\subsection{Invariant theory}
Now suppose $G$ acts on $\rc$ in such a way that $\rc_*$ is the symmetric algebra on a $K$-representation $V$ of $G$. 
If we now assume that $G$ acts by pseudoreflections, the
Shephard-Todd theorem  \cite[7.2.1]{Benson} states that $\rc_*^G$ is also  a polynomial ring 
$$\rc_*^G=K[f_1, \ldots , f_r]$$
where $f_i$ is of degree $e_i$.  Watanabe's Theorem
\cite[4.6.2]{Benson} follows easily from
 the local cohomology of polynomial rings
$$H^r_{\fn}(\rc_*)=\Sigma^{-(d_1+\cdots +d_r)}\rc_*^{\vee} \tensor \det,$$
where $det$ is the determinant of $V$. By Solomon's  Theorem \cite[7.3.1]{Benson}, we have
$$  \Sigma^b (\rc_*\tensor \det)^G =\rc_*^G, $$
where the Solomon Supplement  is $b=(d_1+\cdots +d_r)-(e_1+\cdots +e_r)$. Hence we find 

\begin{itemize}
\item $ \rc_*^G$  (and hence $\rc^{G}$) is Gorenstein of shift $a+b $
\item $ \Cn (\rc)_*^G$  (and hence $\Cn (\rc)^{hG}$)  has Anderson self-duality of shift $a+b+1$
\end{itemize}

Without any restriction on the action, the Hilbert series of the ring
of invariants  may be calculated by character theory from Molien's
Theorem \cite[2.5.2]{Benson}:
$$p(K[V]^G,t)=\frac{1}{|G|}\sum_{g\in G}\frac{1}{\det (1-g^{-1}t,
  V)}. $$

This gives an alternative method for finding the degrees of generators
if the invariants are polynomial.  A more direct route to finding the Gorenstein shift is directly from
the Hilbert series: Stanley's Theorem \cite[Theorem 5.5]{Stanley} shows that if $A$ is a Gorenstein graded ring of shift
$a$, free over $K$
of  Krull dimension $r$, the shift can be deduced from the functional equation 
$$p(A, 1/t)=(-1)^r t^{r-a}p(t), $$
where $r$ is the Krull dimension.

\subsection{Examples}
We make explicit two well-known and rather simple examples. 

\begin{example} ($\rc=ku$, $K=\Z$, $G=C_2$). 
We have $\rc_*=\Z [v]$ with $v$ of degree $2$. This is polynomial, and
hence Gorenstein of shift $a=-3$, and it follows that $ku\lra H\Z$ is
Gorenstein of shift $-3$.  In this case $\fn =(v)$ so that  $\Cn ku\simeq ku[1/v]=KU$.  We can then immediately deduce from
Proposition \ref{prop:GorDAndD} that $KU$ is Anderson self-dual of shift $-2$. 

Now consider connective real $K$-theory  $ko$ with the more
complicated coefficient ring $ko_*=\Z [\eta, \alpha,
\beta]/(2\eta, \eta^3, \alpha^2=4\beta)$ where $|\eta|=1, |\alpha|=4,
|\beta|=8$. To show it is Gorenstein,  we can use the fact that by Wood's theorem
$\Hom_{ko}(ku, ko)\simeq \Sigma^{-2}ku$ and therefore $ko$ is
Gorenstein of shift $-3-2=-5$. Alternatively we can use the fact that
$KO\lra KU$ is Galois and  deduce that $KO=KU^{hC_2}$ is
Gorenstein by descent. One can do this integrally by looking at the
descent spectral sequence, but we will not give details here. (The interested reader could consult \cite{HS}.) From
this we infer that $KO$ is Anderson self-dual of
shift $-4$. The fact that $ko$ has Gorenstein duality of shift $-5$
then follows by Lemma \ref{lem:AndDGorD}. 

For the present we will be satisfied to observe the rational result,
which in particular tells us that the Solomon Supplement  is $-2$.

The action is that  $C_2$ acts to negate $v$ so that $V $ is the sign representation $\tilde{\Z}$. 
We have that
$$H^*_{\fn}(ku_*)=H^1_{\fn}(ku_*)=\Z [v,v^{-1}]/\Z [v]= \Sigma^{-2} v^{-1}\cdot \Z [v]^{\vee},$$ 
and we see $(\rc_*\tensor \det)^G\cong \Sigma^{2}\rc_*^G$, so that $b=-2$.

Rationalisation gives $\rc_*^G=\Q[v]^G=\Q [v^2]$. By inspection this is Gorenstein
of shift $-5$, and we see this is also $a+b$ as predicted above. 
\end{example}

\begin{example} ($\rc=tmf(2)$, $K=\Z_{(3)}$ and  $G=GL_2(3)\cong \Sigma_3$). 
We have $\rc_*=K [x, y]$ with $x $ and $y$ of degree $4$. This is polynomial, and
hence Gorenstein of shift $a=-10$. 

The action is that   $V_K=\ker (K\{ \underline{3}\} \lra K)$, where $K\{ \underline{3}\}$ is the permutation representation associated to
$\underline{3}=\{1,2, 3\}$ with the standard action of $\Sigma_3$ (we
will write $[1], [2], [3]$ for the standard basis). 
Of course 
\begin{align*}
H^*_{\fn}(\rc_*)&=H^2_{\fn}(K[x,y])=K [x,x^{-1}, y,y^{-1}]/(K [x,x^{-1}, y]+K [x, y,y^{-1}])\\
&= \Sigma^{-8} (K[x,y]\tensor \det)^{\vee}.
\end{align*}

We now rationalize to apply the above theory. There are three simple
rational representations, $\eps,
\det$ and $V$, of dimensions 1, 1, and 2, where $V=\ker (\Q \{\underline{3}\} \lra \Q)\cong
V_{K}\tensor \Q$. It is routine to
calculate the decomposition of the symmetric powers into these simple
representations of $\Sigma_3$. Writing $(ijk)$ for $i\eps\oplus j\det\oplus kV$,
the decompositions of the first six symmetric powers of $V_K$ (in degrees 0, 4, 8, 12, 16, 20) are 
$(100), (001), (101), (111), (102), (112)$. The rest follow by the
fact that if the part in degree $4d$ decomposes as $(ijk)$
the part in degree $4d+24$ decomposes as $((i+1)(j+1)(k+1))$. 

The dimension  of the invariants is just the number of copies of
$\eps$ which is thus 
$$101111212222323333\ldots .$$ 
If we take $x=[1]-[2]$
and  $y=[2]-[3]$ it  is easy to find the 
invariants $A=x^2+xy+y^2=N(-xy)$ of degree 8 and 
$B=x^3-y^3-3xy(x+y)/2$ of degree 12, giving
$\rc_*^{\Sigma_3}=K[A,B]$. This (and hence $\rc^{hG}$ rationally) is
Gorenstein of  shift $-22$, and $\Cn  (\rc)^{hG}$ is rationally Anderson
self-dual of shift $-21$. 

On the other hand Solomon's theorem shows that $(\rc_*\tensor
\det)^{\Sigma_3} =\Sigma^{12}\rc_*^{\Sigma_3}$. We note that 
$-b=12=(12-4)+(8-4)$ as expected.
\end{example}

\bibliographystyle{amsalpha}
\bibliography{biblio}
\end{document}